\documentclass[11pt]{article}
\usepackage{amsmath,amsthm,amssymb,mathabx,ifsym,a4wide} 
\usepackage{amscd, mathrsfs, wasysym, color}  
\usepackage{graphicx}
\usepackage{epsfig} 
\usepackage{slashed}  
\usepackage[hidelinks]{hyperref}
\usepackage[most]{tcolorbox} 
\newtheorem{theorem}{Theorem}[section]
\newtheorem{proposition}[theorem]{Proposition}
\newtheorem{lemma}[theorem]{Lemma}

\numberwithin{equation}{section} 
\numberwithin{figure}{section}  
\newcommand \la \langle
\newcommand \ra \rangle

\newcommand \Ecal {\mathcal{E}}

\newcommand \underdel {\underline \partial}

\newcommand \trianglerightNEW \triangleright

\newcommand \phih {\widehat \phi}
\newcommand \wh {\widehat w}
\newcommand \uh {\widehat u}

\newcommand \auth {\textsc}

\newcommand \Kcal {\mathcal K}

\newcommand \bei {\begin{itemize}}
\newcommand \eei {\end{itemize}}
\newcommand \be {\begin{equation}}
\newcommand \bel {\be\label}
\newcommand \ee {\end{equation}}
\newcommand \del \partial
\newcommand \RR {{\mathbb R}}

\newcommand \Hcal {\mathcal H}

\newcommand \eps \epsilon

\newcommand \tildeu {\widetilde u}
\newcommand \tildev {\widetilde v}
\newcommand \tildephi {\widetilde \phi}
\newcommand \tildepsi {\widetilde \psi}

\let\oldmarginpar\marginpar
\renewcommand\marginpar[1]{\-\oldmarginpar[\raggedleft\footnotesize #1]%
{\raggedright\footnotesize #1}}
\begin{document}

\title{\bf The zero mass problem for Klein-Gordon equations} 

\author{Shijie Dong\footnote{
\normalsize Laboratoire Jacques-Louis Lions, Centre National de la Recherche Scientifique, Sorbonne Universit\'e, 
4 Place Jussieu, 75252 Paris, France. 
\newline
Email : {\sl dongs@ljll.math.upmc.fr}
}}

\date{May, 2019}

\maketitle

\begin{abstract}
We are interested in the global solutions to a class of Klein-Gordon equations, and particularly in the unified time decay results with respect to the possibly vanishing mass parameter.
We give for the first time a rigorous proof, which relies on both the flat foliation and the hyperboloidal foliation of the Minkowski spacetime. In order to take advantages of both foliations, an iteration procedure is used.
\end{abstract}

\tableofcontents
 


\section{Introduction}

\subsection*{Model problem and main result}

We are interested in the following system of wave-Klein-Gordon equations
\bel{eq:Va-model}
\aligned
- \Box u + m^2 u
&= M_1 v^3 + P^\alpha \del_\alpha ( v^2),
\\
- \Box v + v
&= N_1 (\del_t u)^2 + N_2 u^3 + N_3 u v,
\endaligned
\ee
in which the range of the mass parameter in the $u$ equation is
$$
m \in [0, 1],
$$
and $M_1, N_1, N_2, N_3, P^\alpha$ are fixed constants which are independent of the mass parameter $m$.
The wave operator is defined by $\Box := \eta^{\alpha \beta} \del_\alpha \del_\beta$, with $\eta = \text{diag}(-1, 1, 1, 1)$ the metric of the Minkowski spacetime. Throughout we use Greek letters to denote spacetime indices and Latin letters to denote space indices, and the Einstein summation convention is adopted unless specified.

The initial data are prescribed at the hypersurface $t=t_0$
\bel{eq:Va-ID}
\big( u, v\big)(t_0, \cdot)
= (u_0, v_0),
\qquad
\big( \del_t u, \del_t v\big)(t_0, \cdot)
= (u_1, v_1).
\ee
In the following we will take $t_0 = 2$, and assume the support of the initial data is in $B(0, 1)$, which is the ball centred at the origin with radius 1.

We are interested in the existence of global solutions to the system \eqref{eq:Va-model} which is uniform in terms of the parameter $m \in [0, 1]$, and particularly, in the behavior of $u$ in the limiting process of $m \to 0$, which is proposed by LeFloch \cite{PLF}. Intuitively, in a short time range the effect of the mass term in the $u$ equation is negligible when the mass parameter $m$ is very tiny, and the solution $u$ is expected to behave more like a wave component, i.e. solution to the $u$ equation with $m = 0$. The goal of this paper is to give a precise statement and provide a rigorous proof on this problem.

The main result is now stated.

\begin{theorem}\label{thm:Va-main}
Consider the  system \eqref{eq:Va-model} with the mass parameter $m \in [0, 1]$,
and let $N \geq 14$ be an integer.  
Then there exists $\eps_0 > 0$, which is notably independent of $m$, such that for all $\eps \in (0, \eps_0)$ and 
all compactly supported initial data $(u_0, u_1, v_0, v_1)$
satisfying the smallness condition 
\bel{eq:3Dmasssmall}
\| u_0, v_0 \|_{H^{N+1}(\RR^3)} + \| u_1, v_1 \|_{H^N(\RR^3)} 
\leq \eps,
\ee
the initial value problem \eqref{eq:Va-model}--\eqref{eq:Va-ID} admits a global-in-time solution $(u, v)$.
Moreover it holds
\bel{eq:Va-unified}
| u(t, x) | 
\lesssim\footnote{We always use $B_1\lesssim B_2$ to denote $B_1 \leq C B_2$ with $C$ a generic constant independent of $m, \eps$ and $C_1$ to be introduced.} 
{1 \over t + m t^{3/2}},
\quad
| v(t, x) |
\lesssim t^{-3/2 + \delta},
\ee 
for arbitrarily small $0<\delta \ll 1/10$.
\end{theorem}

We note that Theorem \ref{thm:Va-main} is much easier to prove at the end points of $m = 0$ and $m = 1$, in which cases the equation of $u$ is a wave equation and a Klein-Gordon equation respectively. But more is involved in obtaining a uniform result in terms of $m \in [0, 1]$ (especially when $m$ is very small but nonzero): 

\bei
\item
The $L^2$-type estimates and sup-norm estimates obtained by the estimates on the mass term $m^2 u$ cannot be used due to the bad factor $m^{-1}$. To be more precise, even if $m \|u\|_{L^2}$ has good bound, the bound of $\|u\|_{L^2}$ might below up when $m \to 0$.
\item
The scaling vector field and the conformal vector field do not commute with the Klein-Gordon operator $- \Box + m^2$.
\item
The tricks in Proposition 3.1 and Proposition 3.3 of \cite{PLF-YM-cmp} in obtaining pointwise estimates for wave or Klein-Gordon components cannot be applied on $u$ component due to its possibly vanishing mass $m$.
\item 
It is hard to get either a good uniform $L^2$-type estimate or a sharp uniform sup-norm estimate on $u$ component in terms of $m$.
\eei

We will only prove Theorem \ref{thm:Va-main} in dimension $3$, but the same proof applies to higher dimensions $\geq 4$ automatically. Thanks to the fast decay of solutions to wave or Klein-Gordon equations in dimensions $\geq 4$, it is harmless to add many quadratic nonlinearities to \eqref{eq:Va-model}.
We also note that the zero mass problem arises in the study of the electroweak standard model, see \cite{DLW}, where the mass of the left-handed neutrino spinor is extremely small but nonzero according to the physical experiments. Moreover, the model problem \eqref{eq:Va-model} can cover the Klein-Gordon-Zakharov model (see for instance \cite{OTT,Tsutaya}), and one refers to \cite{DW} for another generalisation of the study on that model.


\subsection*{Previous work and new ideas}

Let us briefly review some existing work before we illustrate our techniques for treating the system \eqref{eq:Va-model}.
It was shown independently by Klainerman \cite{Klainerman86} and Christodoulou \cite{Christodoulou} that wave equations with null form nonlinearities admit global-in-time solutions. Various important results on the wave equations came out by employing the vector field method by Klainerman, or the conformal method by Christodoulou.
On the other hand, Klainerman \cite{Klainerman85} and Shatah \cite{Shatah} were first able to treat Klein-Gordon equations with quadratic nonlinearities in $\RR^{1+3}$. By relying on Klainerman's vector field method and Shatah's normal form method, many results on the (wave and) Klein-Gordon equations were proved.
Later on, LeFloch-Ma \cite{PLF-YM-book} introduced the hyperboloidal foliation method, which allows one to treat coupled wave and Klein-Gordon equations in one framework.

However, in terms of the unified decay estimates in \eqref{eq:Va-unified}, the only existing such result, to the best knowledge of the author, was proved in \cite{DLW} for the Dirac component.  
Relying on the special structure of the Dirac equation, the authors in \cite{DLW} were able to define a positive definite energy functional adapted to the hyperboloidal foliation. In particular the energy functional is independent of the mass parameter of the Dirac equation, which allows one to obtain the unified decay result.

In this paper, we rely on the hyperboloidal foliation method to show the decay results $t^{-3/2}$ (possibly with a factor of $m^{-1}$) of the Klein-Gordon components, while we utilise the flat foliation and the vector field method to obtain the decay results $t^{-1}$ of $u$ component which are uniform in terms of $m \in [0, 1]$. The key to obtaing the uniform decay results $t^{-1}$ of $u$ component is to prove a uniform $L^2$ bound, which is stated in Proposition \ref{prop:Va-key}, where we analyse $u$ equation in the frequency space; see also Proposition \ref{prop:3DmassH} for the homogeneous case. In order to take advantages of both foliations, we use an iteration method, which was used in \cite{Bachelot,Tsutaya} and many others. The compactness assumption in Theorem \ref{thm:Va-main} is due to the use of the hyperboloidal foliation, which for instance was removed in \cite{PLF-YM-arXiv1,PLF-YM-arXiv2}.


\subsection*{Some special types of quadratic nonlinearities}

We note that there is one quadratic term in the $u$ equation, which is of divergence form. It is interesting to investigate other possibilities of quadratic nonlinearities in the $u$ equation which also lead to global-in-time solution as well as the unified decay results \eqref{eq:Va-unified}. We find it safe to include the following two special types of quadratic terms, while we leave the effects of more general quadratic terms open.

\paragraph{Type 1.}

Denote the strong null forms (see \cite{Georgiev}) by
$$
Q_{\alpha \beta}(u, v)
=
\del_\alpha u \del_\beta v - \del_\beta u \del_\alpha v,
$$
we know that it holds (see for instance \cite{Sogge})
$$
\aligned
Q_{a b}(u, v)
&=
{1\over t} \big( \del_t u \Omega_{ab} v + L_a u \del_b v - L_b u \del_a v \big),
\\
Q_{0 a}
&=
{1\over t} \big( \del_t u L_a v - L_a u \del_t v \big).
\endaligned
$$
It is not hard to see that the good factor $t^{-1}$ makes the null forms $Q_{\alpha \beta}$ behave like a cubic term. 
Hence we can add the strong null forms $Q_{\alpha \beta}$ in the $u$ equation.

\paragraph{Type 2.}

Recall that Delort et. \cite{Delort} introduced one notion of null forms for the Klein-Gordon equations when they studied the global solutions to the Klein-Gordon system with different masses in dimension two. Later on, Katayama et. \cite{KOS} gave a characterization of the null condition proposed in \cite{Delort} and \cite{KaSu}. 
Motivated by their work, we find it also harmless to consider
\be 
\aligned
- \Box u + m^2 u 
&= (2 - m^2 ) v^2 + \del_\alpha v \del^\alpha v,
\\
- \Box v + v
&= Q_v.
\endaligned
\ee
The reason is that the variable
$$
\underline{u}
:= u + v^2
$$
satisfies the equation with cubic nonlinearities, i.e.
\be 
- \Box \underline{u} + m^2 \underline{u}
= 2 v Q_v.
\ee
In conclusion, we can also treat the quadratic term $C (2 - m^2 ) v^2 + C \del_\alpha v \del^\alpha v$ in the $u$ equation when $Q_v$ is sufficiently good, with $C$ a constant independent of $m$.


\subsection*{Organisation of the paper}
In Section \ref{sec:pre}, we introduce some notations and preliminaries. Next in Section \ref{sec:linear}, we study the unified decay estimates for both homogeneous and inhomogeneous linear Klein-Gordon equations with varying mass. Moreover we introduce the solution space and the solution map with several properties in Section \ref{sec:iteration}. Finally in Section \ref{sec:proof} we prove Theorem \ref{thm:Va-main} relying on an iteration process which takes advantages of both the flat and the hyperboloidal foliations. 


\section{Notations and preliminaries}\label{sec:pre}

\subsection{Introduction of the hyperboloidal foliation method}
 
We recall some notations of the hyperboloidal foliation of the Minkowski spacetime used in \cite{PLF-YM-book,PLF-YM-cmp},
which was also introduced and used by Klainerman \cite{Klainerman85} and in the book \cite{Hormander}.
We need to introduce and study the energy functional for wave or Klein-Gordon components on hyperboloids,
and it is necessary to first recall some notations from \cite{PLF-YM-book} concerning the hyperboloidal foliation method. We consider here the $(3+1)$-dimensional Minkowski spacetime with signature $(-, +, +, +)$, and in Cartesian coordinates we adopt the notation of one point $(t, x) = (x^0, x^1, x^2, x^3)$, with its spatial radius $r := | x | = \sqrt{(x^1)^2 + (x^2)^2 + (x^3)^2}$. Partial derivatives are denoted by $\del_\alpha := \del_{x^\alpha}$ (for $\alpha=0, 1, 2, 3$), and 
\be
L_a := x^a \del_t + t \del_a, \qquad a= 1, 2, 3
\ee
represent the Lorentz boosts. Throughout, the functions considered are defined in the interior of the future light cone 
$$
\Kcal
:= \{(t, x): r< t-1 \},
$$
with vertex $(1, 0, 0, 0)$. We denote the hyperboloidal hypersurfaces by 
$$
\Hcal_s
:= \{(t, x): t^2 - r^2 = s^2 \},
\qquad s\geq 2.
$$
We emphasize here that within the cone $\Kcal$ it holds for all points on $\Hcal_s$
\be 
s < t < s^2,
\qquad
r < t.
\ee
Besides, the subsets of $\Kcal$ limited by two hyperboloids $\Hcal_{s_0}$ and $\Hcal_{s_1}$ with $s_0 \leq s_1$ are denoted by
$$
\Kcal_{[s_0, s_1]} 
:= \big\{(t, x): s_0^2 \leq t^2- r^2 \leq s_1^2; \, r<t-1 \big\}.
$$

We now introduce the semi-hyperboloidal frame
\bel{eq:semi-hyper}
\underdel_0
:= \del_t, \qquad \underdel_a:= {L_a \over t} = {x^a\over t}\del_t+ \del_a.
\ee
We notice that the vectors $\underdel_a$ generate the tangent space of the hyperboloids. Besides, the vector field 
$$
\underdel_\perp:= \del_t+ (x^a / t)\del_a
$$ 
is orthogonal to the hyperboloids and is proportional to the scaling vector field $S := t \del_t + x^a \del_a$.

The dual of the semi-hyperboloidal frame is given by 
$$
\underline{\theta}^0
:= dt- (x^a / t)dx^a,
\qquad
\underline{\theta}^a
:= dx^a.
$$ 
The (dual) semi-hyperboloidal frame and the (dual) natural Cartesian frame are connected by the following relations
\bel{semi-hyper-Cts}
\aligned
\underdel_\alpha
&= \Phi_\alpha^{\alpha'}\del_{\alpha'}, 
\qquad 
\del_\alpha
= \Psi_\alpha^{\alpha'}\underdel_{\alpha'}, 
\\
\underline{\theta}^\alpha
&= \Psi^\alpha_{\alpha'}dx^{\alpha'}, 
\qquad 
dx^\alpha
= \Phi^\alpha_{\alpha'}\underline{\theta}^{\alpha'},
\endaligned
\ee
where the transition matrix ($\Phi^\beta_\alpha$) and its inverse ($\Psi^\beta_\alpha$) are given by
\be
(\Phi_\alpha^{ \beta})=
\begin{pmatrix}
1 & 0 &   0 &  0   \\
{x^1 / t} & 1  & 0   &  0  \\
{x^2 / t} &  0  &  1  &  0   \\
{x^3 / t} &  0 & 0   & 1
\end{pmatrix}
\ee
and 
\be
(\Psi_\alpha^{ \beta})=
\begin{pmatrix}
1 & 0 &   0 &  0   \\
-{x^1 / t} & 1  & 0   &  0  \\
-{x^2 / t} &  0  &  1  &  0   \\
-{x^3 / t} &  0 & 0   & 1
\end{pmatrix}.
\ee


\subsection{Energy estimates on hyperboloids}

Following \cite{PLF-YM-cmp} and considering in the Minkowski background, we introduce the energy functional $E_m$ for a nice function $\phi = \phi(t, x)$ defined on the hyperboloid $\Hcal_s$
\bel{eq:2energy} 
\aligned
E_m(s, \phi)
&:=
\int_{\Hcal_s} \Big( \big(\del_t \phi \big)^2+ \sum_a \big(\del_a \phi \big)^2+ 2 (x^a/t) \del_t \phi \del_a \phi + m^2 \phi ^2 \Big) \, dx,
\endaligned
\ee 
which has two other equivalent (and more useful) expressions
\bel{eq:2energy2} 
\aligned
E_m(s, \phi)
&=
\int_{\Hcal_s} \Big( \big( (s/t)\del_t \phi \big)^2+ \sum_a \big(\underdel_a \phi \big)^2+ m^2 \phi^2 \Big) \, dx
\\
&= \int_{\Hcal_s} \Big( \big( \underdel_\perp \phi \big)^2+ \sum_a \big( (s/t)\del_a \phi \big)^2+ \sum_{a<b} \big( t^{-1}\Omega_{ab} \phi \big)^2+ m^2 \phi^2 \Big) \, dx,
\endaligned
\ee
in which 
$$
\Omega_{ab}
:= x^a\del_b- x^b\del_a
$$
are the rotational vector fields, and $\underdel_{\perp}= \del_t+ (x^a / t) \del_a$ is the orthogonal vector field. It is helpful to point it out that each term in the expressions \eqref{eq:2energy2} are non-negative, which is vital in estimating the energies of wave or Klein-Gordon equations.
We use the notation 
$$
E(s, \phi)
:= E_0(s, \phi)
$$ 
for brevity.
In the above, the integral in $L_f^1(\Hcal_s)$ is defined from the standard (flat) metric in $\RR^3$, i.e.
\bel{flat-int}
\|\phi \|_{L^1_f(\Hcal_s)}
:=\int_{\Hcal_s}|\phi | \, dx
=\int_{\RR^3} \big|\phi(\sqrt{s^2+r^2}, x) \big| \, dx.
\ee
By contrast, we will also frequently use the norms of functions on the flat slices, which are denoted by
\be 
\|\phi \|
:=
\| \phi \|_{L^2(\RR^3)}
:=
\Big( \int_{\RR^3} |\phi (t, x)|^2 \, dx \Big)^{1/2},
\ee
and the energy functional on the flat slices
\be 
\Ecal_m (t, \phi)
:=
\sum_\alpha \| \del_\alpha \phi \|^2(t) 
+
m^2 \| \phi \|^2(t).
\ee

Next, we recall the energy estimates for wave-Klein-Gordon equations on the hyperboloids.

\begin{proposition}[Energy estimates for wave-Klein-Gordon equations]
For all $m \geq 0$ and $s \geq 2$, it holds that
\bel{eq:w-EE} 
E_m (s, u)^{1/2}
\leq 
E_m (2, u)^{1/2}
+ \int_2^s \| -\Box u + m^2 u \|_{L^2_f(\Hcal_{s'})} \, ds'
\ee
for all sufficiently regular function $u$, which is defined and supported in the region $\Kcal_{[2, s]}$.
\end{proposition}

For the proof, one refers to \cite{PLF-YM-cmp}.


\subsection{Useful inequalities}

\paragraph{Klainerman-Sobolev inequality}

We first state the following Klainerman-Sobolev inequality, whose proof can be found in \cite{Klainerman2}.

\begin{proposition}
Let $u = u(t, x)$ be a sufficiently smooth function which is compactly supported for each fixed $t \geq 2$.
Then for any $t \geq 2$, $x \in \RR^3$, we have
\bel{eq:Va-K-S}
|u(t, x)|
\lesssim t^{-1} \sup_{0\leq t' \leq 2t, |I| \leq 3} \big\| \Gamma^I u \big\|_{L^2(\RR^3)},
\qquad
\Gamma \in A := \{ L_a, \del_\alpha, \Omega_{ab} = x^a \del_b - x^b \del_a \}.
\ee
\end{proposition}

We note that the importance of the Klainerman-Sobolev inequality to our problem is that the scaling vector field $L_0 = t \del_t + x^a \del_a$, which does not commute with $-\Box + m^2$, is not needed.

\paragraph{Sobolev-type inequality on the hyperboloids}

Following from \cite{PLF-YM-book},
we now introduce a Sobolev-type inequality adapted to the hyperboloids, which is important in obtaining the sup-norm estimates for both wave and Klein-Gordon components.

\begin{proposition} \label{prop:sobolev}
For all sufficiently smooth functions $u= u(t, x)$ supported in $\{(t, x): |x|< t - 1\}$ and for all  $s \geq 2$, one has 
\bel{eq:Sobolev2}
\sup_{\Hcal_s} \big| t^{3/2} u(t, x) \big|  
\lesssim \sum_{| J |\leq 2} \| L^J u \|_{L^2_f(\Hcal_s)},
\ee
in which the symbol $L$ denotes the Lorentz boosts. 
\end{proposition}

Following from the Sobolev inequality \eqref{eq:Sobolev2} and the commutator estimates, we have the following inequality
\be
\sup_{\Hcal_s} \big| s \hskip0.03cm t^{1/2} u(t, x) \big|  
\lesssim \sum_{| J |\leq 2} \| (s/t) L^J u \|_{L^2_f(\Hcal_s)}.
\ee

\paragraph{Hardy inequality}

\begin{proposition}
Let $\phi = \phi(x)$ be a sufficiently smooth function in dimensions $\geq 3$, then it holds
\be 
\big\| r^{-1} \phi \big\|
\leq 
C \sum_a \big\| \del_a \phi \big\|.
\ee
\end{proposition}

One also has Hardy inequality adapted to the hyperboloids, see for instance \cite{PLF-YM-cmp}, where one replaces $\| \cdot \|$ by $\| \cdot \|_{L^2_f}$ and $\del_a$ by $\underdel_a$.


\subsection{Commutator estimates}

We now recall some well-known facts about the commutators among different vector fields. 

\begin{proposition}
The following relations are valid
\bel{eq:Va-commutator}
\aligned
&\hskip0.13cm
[\del_\alpha, L_a]
= \delta_{0\alpha} \del_a + \delta_{a\alpha} \del_t,
\qquad
[L_a, \Omega_{bc}]
=\delta_{ab} L_c - \delta_{ac} L_b,
\\
&[\del_\alpha, \Omega_{ab}]
= \delta_{b\alpha} \del_a - \delta_{a \alpha} \del_b,
\qquad
[\Gamma, -\Box + m^2]
=0,
\endaligned
\ee
for all $\Gamma \in A = \{ L_a, \del_\alpha, \Omega_{ab} = x^a \del_b - x^b \del_a \}$, with $\delta_{\alpha \beta}$ the Kronecker delta.
\end{proposition}


\section{Unified decay results for linear Klein-Gordon equations}\label{sec:linear}

\subsection{The homogeneous case}

We first consider a simple homogeneous Klein-Gordon equation with $m \in [0, 1]$
\bel{eq:3DmassH} 
\aligned
- \Box w + m^2 w 
&= 0,
\\
\big( w, \del_t w \big) (0, \cdot)
&=
\big( w_0, w_1 \big),
\endaligned
\ee
and prove the following theorem.

\begin{proposition}[Unified decay results for homogeneous Klein-Gordon equations]
\label{prop:3DmassH}
Consider the initial value problem \eqref{eq:3DmassH}, and assume the initial data are compactly supported and satisfy
\be 
\|w_0\|_{H^5(\RR^3)} + \|w_1 \|_{H^4(\RR^3)}
\leq \eps,
\ee
then the following unified decay result is valid
\bel{eq:3DmassHw} 
|w|
\lesssim \eps \min \{ (t + 2)^{-1}, m^{-1} (t + 2)^{-3/2} \}.
\ee
\end{proposition}

The proof relies on a simple utilization of the Fourier method, which is from the lecture note by Luk \cite{Luk} in treating homogeneous wave equations.
We first revisit some basics in Fourier analysis before giving the proof.

Recall the Fourier transform of a nice function $\phi = \phi(x)$ is defined by
$$
\phih(\xi)
:=
\int_{\RR^3} \phi(x) e^{- 2 \pi i x \cdot \xi} \, dx,
$$
and the inverse Fourier transform of a nice function $\psi = \psi(\xi)$ is defined by
$$
\widecheck{\psi}(x)
:=
\int_{\RR^3} \psi(\xi) e^{2 \pi i x \cdot \xi} \, d\xi.
$$
Next we recall some basic but important facts in Fourier analysis.

\begin{proposition}
The following properties hold for a nice function $\phi = \phi(x)$:
\bei
\item Inverse formula.
\be 
\phi
= \widecheck{\phih}.
\ee
\item Relation between partial derivatives and Fourier multipliers.
\be 
\widehat{\del_a \phi} (\xi)
= 2\pi i \xi_a \phih(\xi).
\ee
\item Plancheral identity.
\be
\| \phi \|_{L^2(\RR^3)}
= \big\| \phih \big\|_{L^2(\RR^3)}.
\ee
\eei
\end{proposition}

\begin{proof}[Proof of Proposition \ref{prop:3DmassH}]
In the Fourier space, the equation \eqref{eq:3DmassH} can be written as
$$
\del_t \del_t \wh(t, \xi) + \xi_m^2 \wh(t, \xi)
= 0,
$$
with initial data
$$
\big( \wh, \del_t \wh \big) (0, \cdot)
=
\big( \wh_0, \wh_1 \big),
$$ 
in which we used the notation 
$$
\xi_m 
:= \big( 4 \pi^2 |\xi|^2 + m^2 \big)^{1/2}.
$$
Next by solving the ordinary differential equation above, we get the explicit solution in the Fourier space
\be 
\wh(t, \xi)
=
\cos(2 \pi t \xi_m) \wh_0(\xi)
+ {\sin(2 \pi t \xi_m) \over 2 \pi \xi_m} \wh_1(\xi),
\ee
which can also be expressed by the following four terms.
\be 
\wh(t, \xi)
= 
e^{2\pi i t \xi_m} \Big( {\wh_0(\xi) \over 2} + {\wh_1(\xi) \over 2\pi i \xi_m} \Big)
+ e^{2\pi i t \xi_m} \Big( {\wh_0(\xi) \over 2} + {\wh_1(\xi) \over 2\pi i \xi_m} \Big).
\ee

Then we estimate the inverse Fourier transform of those four terms above, but we notice that it suffices to estimate the first two terms.
We denote by the inverse Fourier transform of the second term 
$$
I_1
:= \int_{\RR^3} e^{2\pi i (t \xi_m + x \cdot \xi)} {\wh_1(\xi) \over 2\pi i \xi_m} \, d\xi.
$$
Without loss of any generality, we assume 
$$
x
= (0, 0, |x|),
$$
and we use the polar coordinates for the first two components of $\xi$, i.e.
$$
\xi
= \big( \rho \cos \xi_\theta, \rho \sin \xi_\theta, \xi_3 \big),
\qquad
(\rho, \xi_\theta) \in [0, +\infty) \times [0, 2\pi),
$$
and thus
$$
d\xi
= \rho d\rho d\xi_\theta d\xi_3.
$$
It also helps to note that
$$
{\del |\xi| \over \del \rho}
= {\rho \over |\xi|},
\qquad
{\del \xi_m \over \del \rho}
= {\rho \over \xi_m},
$$
as well as
$$
\del_\rho e^{2\pi i t \xi_m}
=
2\pi i t {\rho \over \xi_m} e^{2\pi i t \xi_m}.
$$
Now relying on these results we further arrive at
$$
\aligned
I_1
&=
-{1\over 4 \pi^2  t} \int_{\RR} \int_0^{2\pi} \int_0^{+\infty} \del_\rho e^{2\pi i t \xi_m} e^{2\pi i x\cdot \xi} \wh_1 \, d\rho d\xi_\theta d\xi_3
\\
&=
{1\over 4 \pi^2  t} \int_{\RR} \int_0^{2\pi} \int_0^{+\infty} e^{2\pi i (t \xi_m + x\cdot \xi)} \del_\rho \wh_1 \, d\rho d\xi_\theta d\xi_3
- {1\over 2\pi t} \int_{\RR} e^{2\pi i (t \xi_m + |x| \xi_3)} \wh_1(\rho=0, \xi_3) \, d\xi_3
\\
&=: I_{11} + I_{12},
\endaligned
$$
where we did integration by parts in the second step. 
Observe that 
$$
\int_{\RR} \int_0^{2\pi} \int_0^{+\infty} {1 \over (1 + |\xi|^2)^2} \, d\rho d\xi_\theta d\xi_3
\lesssim 1,
$$
as well as 
$$
\aligned
\sum_a  (1 + |\xi|^2)^2 |\del_a \wh_1 | 
&\lesssim \sum_a \| (1 - \Delta)^2 (x_a w_1) \|_{L^1(\RR^3)}
\\
&\leq \| w_1 \|_{H^4(\RR^3)},
\endaligned
$$
where we used the fact in the last step that
$$
L^p(\Omega) \subset L^1(\Omega),
\qquad
p \geq 1,
$$
when $\Omega \subset \RR^d$ is a compact set. Thus we arrive at
$$
|I_{11}|
\lesssim t^{-1} \| w_1 \|_{H^4(\RR^3)},
$$
and similarly we can show 
$$
|I_{12}|
\lesssim t^{-1} \| w_1 \|_{H^4(\RR^3)}.
$$
To conclude, we have 
\bel{eq:3DmassI1}
|I_1|
\lesssim t^{-1} \| w_1 \|_{H^4(\RR^3)}.
\ee

Next we do the same analysis on the inverse Fourier transform of the first term, which we denote by
$$
I_0
:=
\int_{\RR^3} e^{2\pi i (t \xi_m + x \cdot \xi)} {\wh_0(\xi) \over 2} \, d\xi.
$$
By adopting the same setting, we proceed and get
$$
\aligned
I_0
&=
{1\over 4\pi i t} \int_{\RR} \int_0^{2\pi} \int_0^{+\infty} \del_\rho e^{2\pi i t \xi_m} \xi_m e^{2\pi i x\cdot \xi} \wh_0 \, d\rho d\xi_\theta d\xi_3
\\
&=
- {1\over 4\pi i t} \int_{\RR} \int_0^{2\pi} \int_0^{+\infty} e^{2\pi i t \xi_m} {\rho \over \xi_m} e^{2\pi i x\cdot \xi} \wh_0 \, d\rho d\xi_\theta d\xi_3
\\
&- {1\over 4\pi i t} \int_{\RR} \int_0^{2\pi} \int_0^{+\infty} e^{2\pi i t \xi_m} \xi_m e^{2\pi i x\cdot \xi} \del_\rho \wh_0 \, d\rho d\xi_\theta d\xi_3
\\
&- {1\over 2 i t} \int_{\RR} \big( \xi_3^2 + m^2 \big)^{1/2} e^{2\pi i (t \xi_m + |x| \xi_3)} \wh_0(\rho=0, \xi_3) \, d\xi_3.
\endaligned
$$
Similarly, we conclude that
\bel{eq:3DmassI0}
|I_0|
\lesssim t^{-1} \| w_0 \|_{H^4(\RR^3)}.
\ee 

A combination of \eqref{eq:3DmassI1} and \eqref{eq:3DmassI0} gives
\be
|w(t, x)|
\lesssim (t + 2)^{-1} \big( \| w_0 \|_{H^4(\RR^3)} + \| w_1 \|_{H^4(\RR^3)} \big),
\qquad
t \geq 2.
\ee
On the other hand, we observe that it is easy to show 
\be 
| w(t, x) |
\lesssim \| w_0 \|_{H^4(\RR^3)} + \| w_1 \|_{H^4(\RR^3)},
\qquad
0 \leq t \leq 2.
\ee
Hence we arrive at \eqref{eq:3DmassHw} since the bound $m^{-1} (t + 2)^{-3/2}$ is trivial to prove.
\end{proof}


\subsection{The inhomogeneous case}

\begin{proposition}\label{prop:Va-key}
Consider the wave-Klein-Gordon equation
$$
-\Box u + m^2 u 
= f,
\qquad
\big( u, \del_t u \big)(t_0)
= (u_0, u_1),
$$
with mass $m \in [0, 1]$, and assume 
$$
\| u_0 \|_{L^2(\RR^3)}
+
\| x u_1 \|_{L^2(\RR^3)}
\lesssim C_{t_0},
\quad
\| x f \|_{L^2(\RR^3)}
\leq C_f t^{-1 + q},
$$
for some numbers $C_{t_0}$ and $C_f$.
Then we have
\begin{eqnarray}
\| u \|_{L^2(\RR^3)}
\lesssim 
\left\{
\begin{array}{lll}
C_{t_0} + C_f t^q, & \quad q>0,
\\
C_{t_0} + C_f \log t, & \quad q = 0,
\\
C_{t_0} + C_f, & \quad q<0.
\end{array}
\right.
\end{eqnarray}
\end{proposition}
\begin{proof}
We write the $u$ equation and solution in the Fourier space $(t, \xi)$:
$$
\del_t \del_t \uh + \xi_m^2 \uh
= \widehat{f},
$$
$$
\uh(t, \xi)
= \cos\big( t \xi_m \big) \uh_0
+ { \sin\big( t \xi_m \big) \over \xi_m } \uh_1
+ {1\over \xi_m} \int_{t_0}^{t} \sin\big( (t-t') \xi_m \big) \widehat{f}(t') \, dt'.
$$
in which
$$
\uh_0 = \widehat{u_0},
\qquad
\uh_1 = \widehat{u_1},
\qquad
\xi_m = \sqrt{4 \pi^2 |\xi|^2 + m^2} \geq |\xi|.
$$

Next by the fact $|\sin p|,  |\cos p| \leq 1$, we have the $L^2$ norm estimates 
$$
\aligned
&\quad
\| u \|_{L^2(\RR^3)}
\\
&= 
\| \uh \|_{L^2(\RR^3)}
\leq
\| \uh_0 \|_{L^2(\RR^3)}
+
\| \uh_1 / \xi_m \|_{L^2(\RR^3)}
+
\int_{t_0}^t \big\| \widehat{f} / \xi_m \big\|_{L^2(\RR^3)} (t') \, dt'
\\
&\leq 
\| \uh_0 \|_{L^2(\RR^3)}
+
\| \uh_1 / |\xi| \|_{L^2(\RR^3)}
+
\int_{t_0}^t \big\| \widehat{f} / |\xi| \big\|_{L^2(\RR^3)} (t') \, dt',
\endaligned
$$
where we use the fact $|\xi| \leq \xi_m$ in the last step.

An application of the Hardy inequality in the frequency space gives
$$
\aligned
\| u \|_{L^2(\RR^3)}
&\lesssim 
\| \uh_0 \|_{L^2(\RR^3)}
+
\| \del_{\xi} \uh_1 \|_{L^2(\RR^3)}
+
\int_{t_0}^t \big\| \del_\xi \widehat{f} \big\|_{L^2(\RR^3)} (t') \, dt'
\\
&\lesssim 
\| u_0 \|_{L^2(\RR^3)}
+
\| x u_1 \|_{L^2(\RR^3)}
+
\int_{t_0}^t \big\| x f \big\|_{L^2(\RR^3)} (t') \, dt'.
\endaligned
$$
The proof is complete by recalling the assumptions on $f$ and the basic calculations.
\end{proof}


\section{Setting of the iteration process}\label{sec:iteration}

\subsection{The solution space and the solution map}

The goal of this section is to design a proper solution space $X$, with a solution map $T : X \to X$. We will show that the map $T$ is a contraction map by carefully choosing parameters in the space $X$.

We now introduce the $X$-norm of a sufficiently regular function set $(\phi, \psi) = \big(\phi(t, x), \psi(t, x) \big)$, which is defined by
\bel{eq:Va-Xnorm}
\aligned
\| (\phi, \psi) \|_X
&:=
\sup_{t \geq 2} \sum_{|I| \leq N, \Gamma \in A} \Big(  \Ecal_1(t, \Gamma^I \phi)  + t^{-\delta} \Ecal_1(t, \Gamma^I \psi)  
             \Big)
\\
&             + \sup_{s \geq 2} \sum_{|J| \leq N-5, \Gamma \in A} \bigg( E_m (s, \Gamma^J \phi)^{1/2} + s^{-\delta} E_1 (s, \Gamma^J \psi)^{1/2}\bigg),
\endaligned
\ee
in which $0< \delta \ll 1/10$ and $C_1 \gg 1$ are some constants to be determined, which are fixed once and for all.

Taking the initial data to the model problem \eqref{eq:Va-model} into account, we are now ready to introduce the solution space
\bel{eq:Va-Xspace}
\aligned
X :=
\Big\{ \big(u(t, x), v(t,x)\big):
\big( u, v\big)(t_0, \cdot)
= (u_0, v_0),
\big( \del_t u, \del_t v\big)(t_0, \cdot)
= (u_1, v_1), 
\| (u, v)\|_X \leq C_1 \eps \Big\},
\endaligned
\ee
in which the same $C_1 \gg 1$ is some constant to be determined, and $\eps$ is the size of the initial data. It is not hard to see that the function space $X$ is complete with respect to the metric $\| \cdot \|_X$.

Naturally, the image $T(u, v) = (\phi, \psi)$ of $(u, v) \in X$ is defined as the solution to the linear Klein-Gordon equations
\be
\aligned
- \Box \phi + m^2 \phi
&= M_1 v^3 + P^\alpha \del_\alpha ( v^2),
\\
- \Box \psi + \psi
&= N_1 (\del_t u)^2 + N_2 u^3 + N_3 u v,
\\
\big( \phi, \psi \big)(t_0, \cdot)
= (u_0, & v_0),
\qquad
\big( \del_t \phi, \del_t \psi \big)(t_0, \cdot)
= (u_1, v_1).
\endaligned
\ee

The main task in this section is to prove the following proposition.

\begin{proposition}\label{prop:Va-contraction0}
With properly chosen parameters $\eps, \delta, C_1$, the solution map $T$ satisfies the following contraction property
\be 
\| T(u, v) \|_X
\leq 
{1\over 2} C_1 \epsilon 
~for ~all ~ (u, v) \in X.
\ee
\end{proposition}

On one hand, Proposition \ref{prop:Va-contraction0} ensures that $T(X) \subset X$. On the other hand, the proof of Proposition \ref{prop:Va-contraction0} can also be adapted to prove the solution map $T$ is a contraction map. A combination of these two allows us to rely on the fixed point theorem to prove the existence of global solutions to \eqref{eq:Va-model}.


\subsection{Proof of Proposition \ref{prop:Va-contraction0}}

In order to prove Proposition \ref{prop:Va-contraction0}, we need to rely on a few lemmas and propositions given below.

\begin{lemma}\label{lem:Va-decay1}
Let $(u, v) \in X$, then it holds for all $\Gamma \in A = \{ L_a, \del_\alpha, \Omega_{ab}\}$ that
\be 
\aligned
m t^{3/2} |\Gamma^K u| + t^{1/2} s |\del \Gamma^K u|
&\leq C C_1 \eps,
\qquad
|K| \leq N-7,
\\
t^{3/2} s^{-\delta} |\Gamma^K v| + t^{1/2} s^{1-\delta} |\del \Gamma^K v|
&\leq C C_1 \eps,
\qquad
|K| \leq N-7,
\endaligned
\ee
as well as
\be 
\aligned
t |\Gamma^J u| + t |\del \Gamma^J u| + t^{1 - \delta} |\Gamma^J v|
&\leq C C_1 \eps,
\qquad
|J| \leq N-5,
\endaligned
\ee
\end{lemma}
\begin{proof}
The proof of the first two estimates follows from the Sobolev-type inequality on the hyperboloids in Proposition \ref{prop:sobolev}, as well as the commutator estimates.

For the last estimate,  it follows from the Klainerman-Sobolev inequality \eqref{eq:Va-K-S} and the commutator estimates.
\end{proof}

\begin{proposition}[Energy estimates on the flat slices]
\label{prop:Va-EEflat}
Assume $(\phi, \psi) = T(u, v)$ with $(u, v) \in X$, then for all $\Gamma \in A$ we have
\bel{eq:Va-EEflat} 
\aligned
\Ecal_1 (t, \Gamma^I \phi)^{1/2}
&\leq C \eps + C (C_1 \eps)^2,
\qquad
|I| \leq N,
\\
\Ecal_1 (t, \Gamma^I \psi)^{1/2}
&\leq C \eps + C (C_1 \eps)^2 t^\delta,
\qquad
|I| \leq N.
\endaligned
\ee
\end{proposition}
\begin{proof}
First, by the energy estimates it holds for all $|I| \leq N$ that
$$
\aligned
\Ecal_m (t, \Gamma^I \phi)^{1/2}
&\leq 
\Ecal_m (2, \Gamma^I \phi)^{1/2}
+ \int_2^t \big\| \Gamma^I \big(M_1 v^3 + P^\alpha \del_\alpha ( v^2)\big) \big\|(t') \, dt'
\\
&\leq 
\Ecal_m (2, \Gamma^I \phi)^{1/2}
+ C \sum_{I_1 + I_2 = I, |I_1| \geq |I_2|, \alpha} \int_2^t \big\| \Gamma^{I_1} \del_\alpha v \big\|  \big\|\Gamma^{I_2} v \big\|_{L^\infty} \, dt'.
\endaligned
$$
We insert the estimates of $v$ and arrive at
$$
\aligned
\Ecal_m (t, \Gamma^I \phi)^{1/2}
&\leq 
\Ecal_m (2, \Gamma^I \phi)^{1/2}
+ C (C_1 \eps)^2 \int_2^t t'^\delta t'^{-3/2 + \delta}
\\
&\leq
\eps + C (C_1 \eps)^2.
\endaligned
$$

Next, we need to estimate the $L^2$ norm of $\phi$ component. Following \cite{Katayama12a} and for all $|I| \leq N$, we observe that
$$
\aligned
\| \Gamma^I \phi \|
\leq 
\| \Gamma^I \del_\alpha \Phi^\alpha \| 
+
\| \Gamma^I \Phi^5 \|, 
\endaligned
$$
in which $\Phi^\alpha, \Phi^5$ are solutions to the following equations:
\be 
\aligned
- \Box \Phi^\alpha + m^2 \Phi^\alpha
&=
P^\alpha v^2,
\\
\big( \Phi^\alpha, \del_t \Phi^\alpha \big)(2, \cdot)
&= (0, 0),
\endaligned
\ee

\be 
\aligned
- \Box \Phi^5 + m^2 \Phi^5
&=
M_1 v^3,
\\
\big( \Phi^5, \del_t \Phi^5 \big)(2, \cdot)
&= (u_0, u_1 - P^0 v_0^2).
\endaligned
\ee
We notice that
$$
\| \Gamma^I \del_\alpha \Phi^\alpha \| 
\leq C \sum_{|I_1| \leq N} \| \del_\alpha \Gamma^{I_1} \Phi^\alpha \| 
\leq C\eps + C (C_1 \eps)^2, 
$$
which follows from an energy estimate, and
$$
\| \Gamma^I \Phi^5 \|
\leq C \eps + C (C_1 \eps)^2,
$$
which is thanks to the estimate
$$
\big\| \Gamma^I (v^3) \big\|
\leq C (C_1 \eps)^3 t^{-5/2},
$$
and Proposition \ref{prop:Va-key}. Hence together with what we have proved for $\Ecal_m(t, \Gamma^I \phi)^{1/2}$, the first estimate in \eqref{eq:Va-EEflat} is now obtained.

Finally we turn to the estimates of $\psi$ component. For $|I| \leq N$, the energy estimates give us 
$$
\aligned
\Ecal_1 (t, \Gamma^I \psi)^{1/2}
&\leq 
\Ecal_1 (2, \Gamma^I \psi)^{1/2}
+ 
\int_2^t \big\| \Gamma^I \big(N_1 (\del_t u)^2 + N_2 u^3 + N_3 u v \big) \big\|(t') \, dt'
\\
&\leq 
\Ecal_1 (2, \Gamma^I \psi)^{1/2}
+ 
C \sum_{I_1 + I_2 = I, |I_1| \leq |I_2|} \int_2^t \Big(\big\| \Gamma^{I_1} \del_t u \big\|_{L^\infty} \big\| \Gamma^{I_2} \del_t u \big\| 
+ \big\| \Gamma^{I_1} u^2 \big\|_{L^\infty} \big\| \Gamma^{I_2} u \big\| 
\\
&\hskip4.5cm
+ \big\|\Gamma^{I_1} v \big\|_{L^\infty} \big\| \Gamma^{I_2} u \big\| 
+ \big\|\Gamma^{I_1} u \big\|_{L^\infty} \big\| \Gamma^{I_2} v \big\|  \Big) \, dt'.
\endaligned
$$
Since $(u, v) \in X$, we thus have
$$
\aligned
\Ecal_1 (t, \Gamma^J \psi)^{1/2}
&\leq 
\Ecal_1 (2, \Gamma^J \psi)^{1/2}
+ 
C (C_1 \eps)^2 \int_2^t t'^{-1+\delta} \, dt'
\\
&\leq 
\eps + C (C_1 \eps)^2 t^\delta.
\endaligned
$$
\end{proof}

By the local estimates of the solution $(\phi, \psi)$, we have the following bounds of its hyperboloidal energy on the initial slice.

\begin{lemma}
Let $(\phi, \psi) = T(u, v)$ with $(u, v) \in X$, then for all $|J| \leq N-5$ and $\Gamma \in A$ it holds true that
\be 
E_m (2, \Gamma^J \phi)^{1/2}
+
E_1 (2, \Gamma^J \psi)^{1/2}
\leq C \eps + C (C_1 \eps)^2.
\ee
\end{lemma}

\begin{proposition}[Energy estimates on the hyperboloids]
\label{prop:Va-EEhyper}
Assume $(\phi, \psi) = T(u, v)$ with $(u, v) \in X$, then we have for all $\Gamma \in A$ that
\be
\aligned
E_m (s, \Gamma^J \phi)^{1/2}
&\leq C \eps + C (C_1 \eps)^2,
\qquad
|J| \leq N-5,
\\
E_1 (s, \Gamma^J \psi)^{1/2}
&\leq C \eps + C (C_1 \eps)^2 s^\delta,
\qquad
|J| \leq N-5.
\endaligned
\ee
\end{proposition}
\begin{proof}
From the energy estimates, we have for all $|J| \leq N-5$
$$
\aligned
E_m (s, \Gamma^J \phi)^{1/2}
&\leq
E_m (2, \Gamma^J \phi)^{1/2}
+
\int_2^s \big\| \Gamma^J \big(M_1 v^3 + P^\alpha \del_\alpha ( v^2)\big) \big\|_{L^2_f(\Hcal_{s'})} \, ds'
\\
&\leq 
E_m (2, \Gamma^J \phi)^{1/2}
+ C \sum_{J_1 + J_2 = J, |J_1| \geq |J_2|, \alpha} \int_2^s \Big(  \big\| \Gamma^{J_1} v \big\|_{L^2_f(\Hcal_{s'})}  \big\| \Gamma^{J_2} \del_\alpha v \big\|_{L^\infty(\Hcal_{s'})} 
\\
&\hskip3cm
+ \big\| (s'/t) \Gamma^{J_1} \del_\alpha v \big\|_{L^2_f(\Hcal_{s'})}  \big\|(t/s') \Gamma^{J_2} v \big\|_{L^\infty(\Hcal_{s'})} \Big) \, ds'.
\endaligned
$$
Successively, we get
$$
\aligned
E_m (s, \Gamma^J \phi)^{1/2}
&\leq
E_m (2, \Gamma^J \phi)^{1/2}
+
C (C_1 \eps)^2 \int_2^s s'^{-3/2 + 3\delta} \, ds'
\\
&\leq
C \eps + C (C_1 \eps)^2.
\endaligned
$$

In the process of estimating $\| \Gamma^J (u v) \|_{L^2_f(\Hcal_s)}$, we always take $L^2$ norm of $v$ and take $L^\infty$ norm of $u$. To be more precise, we have
$$
\aligned
E_1 (s, \Gamma^J \psi)^{1/2}
&\leq
E_1 (2, \Gamma^J \psi)^{1/2}
+
\int_2^s \big\| \Gamma^J \big(N_1 (\del_t u)^2 + N_2 u^3 + N_3 u v \big) \big\|_{L^2_f(\Hcal_{s'})} \, ds'
\\
&\leq 
E_1 (2, \Gamma^J \psi)^{1/2}
+ 
C \sum_{J_1 + J_2 = I, |J_1| \leq |J_2|} \int_2^s \Big(\big\| (t/s') \Gamma^{J_1} \del_t u \big\|_{L^\infty(\Hcal_{s'})} \big\| (s'/t) \Gamma^{J_2} \del_t u \big\|_{L^2_f(\Hcal_{s'})} 
\\
&\hskip3cm
+ \big\| r \Gamma^{J_1} u^2 \big\|_{L^\infty(\Hcal_{s'})} \big\| r^{-1} \Gamma^{J_2} u \big\|_{L^2_f(\Hcal_{s'})} 
+ \big\|\Gamma^{J_1} u \big\|_{L^\infty(\Hcal_{s'})} \big\| \Gamma^{J_2} v \big\|_{L^2_f(\Hcal_{s'})} 
\\
&\hskip3cm
+ \big\|\Gamma^{J_2} u \big\|_{L^\infty(\Hcal_{s'})} \big\| \Gamma^{J_1} v \big\|_{L^2_f(\Hcal_{s'})} \Big) \, ds',
\endaligned
$$
which leads us to 
$$
\aligned
E_1 (s, \Gamma^J \psi)^{1/2}
&\leq
E_1 (2, \Gamma^J \psi)^{1/2}
+
C (C_1 \eps)^2 \int_2^s s'^{-1 + \delta} \, ds'
\\
&\leq
C \eps + C (C_1 \eps)^2 s^\delta.
\endaligned
$$
Thus the proof is complete.
\end{proof}

With the preparations above, we are ready to give the proof of Proposition \ref{prop:Va-contraction0}.

\begin{proof}[Proof of Proposition \ref{prop:Va-contraction0}]
We choose $C_1 > 0$ large enough such that $C_1^{1/4} \geq 2 C + 1$ for all generic constants $C$ appearing in the analysis, and choose $\eps > 0$ very small such that $C_1^2 \eps \leq \delta \ll 1/10$. Then by recalling Proposition \ref{prop:Va-EEflat} and Proposition \ref{prop:Va-EEhyper}, we easily get 
$$
\|T(u, v) \|_X 
\leq {1\over 2} \| (u, v) \|_X
\quad
for~all~(u, v) ~\in X,
$$
which proves Proposition \ref{prop:Va-contraction0}.

\end{proof}


\section{Proof of Theorem \ref{thm:Va-main}}\label{sec:proof}

\subsection{Proof of the contraction map}

The main goal in this section is to prove the following proposition.

\begin{proposition}
\label{prop:Va-contraction1}
Let $(u, v), (\tildeu, \tildev) \in X$, and $(\phi, \psi) = T(u, v), (\tildephi, \tildepsi) = T(\tildeu, \tildev)$. 
If we denote 
\be 
h 
=
\big\| (u- \tildeu, v- \tildev) \big\|_X,
\ee
then it is true that
\be 
\big\| (\phi - \tildephi, \psi - \tildepsi) \big\|_X
\leq {1\over 2} h.
\ee
\end{proposition}

Before we proceed, we first write the equations satisfied by the difference 
$$ 
(\chi_1, \chi_2)
:=
\big(\phi - \tildephi, \psi - \tildepsi \big),
$$
and we find
\bel{eq:Va-diff} 
\aligned
- \Box \chi_1 + m^2 \chi_1
&=
M_1 (v - \tildev) (v^2 + v \tildev + \tildev^2)
+ P^\alpha \del_\alpha \big( (v - \tildev) (v + \tildev) \big),
\\
- \Box \chi_2 + \chi_2
&=
N_1 (\del_t u - \del_t \tildeu) (\del_t u + \del_t \tildeu)
+ N_2 (u - \tildeu) (u^2 + u \tildeu + \tildeu^2)
\\
&
+ N_3 (v - \tildev) u + N_3 (u - \tildeu) \tildev,
\\
\big(\chi_1, \chi_2\big) (2, \cdot)
&= (0, 0),
\qquad
\big(\del_t \chi_1, \del_t \chi_2 \big)(2, \cdot)
= (0, 0).
\endaligned
\ee

We have several useful observations illustrated in the following lemmas. The one right below gives us the pointwise decay results of the difference variable $(u - \tildeu, v - \tildev)$.

\begin{lemma}
With the same notations and assumptions as Proposition \ref{prop:Va-contraction1}, we have the following pointwise estimates true for all $\Gamma \in A$
\be 
\aligned
m t^{3/2} \big|\Gamma^K (u - \tildeu) \big| + t^{1/2} s \big|\del \Gamma^K (u - \tildeu) \big|
&\leq C h,
\qquad
|K| \leq N-7,
\\
t^{3/2} s^{-\delta} \big|\Gamma^K (v - \tildev) \big| + t^{1/2} s^{1-\delta} \big|\del \Gamma^K (v - \tildev) \big|
&\leq C h,
\qquad
|K| \leq N-7,
\endaligned
\ee
as well as
\be 
\aligned
t \big|\Gamma^J (u - \tildeu) \big| + t \big|\del \Gamma^J (u - \tildeu) \big| + t^{1 - \delta} \big|\Gamma^J (v - \tildev) \big|
&\leq C h,
\qquad
|J| \leq N-5.
\endaligned
\ee
\end{lemma}
\begin{proof}
The proof of Lemma \ref{lem:Va-decay1} also applies here.
\end{proof}

\begin{lemma}
Consider the equations in \eqref{eq:Va-diff} with the same assumptions as Proposition \ref{prop:Va-contraction1}, 
then initially it holds
\be 
\Ecal_1(2, \Gamma^I \chi_1)^{1/2}
+
\Ecal_1(2, \Gamma^I \chi_2)^{1/2}
\leq
C C_1 \eps h,
\qquad
|I| \leq N, 
\qquad
\Gamma \in A.
\ee
\end{lemma}

Furthermore, by the local estimates of the solution $(\chi_1, \chi_2)$, we have the following bounds of its hyperboloidal energy on the initial slice.

\begin{lemma}
For all $|J| \leq N-5$ and $\Gamma \in A$ it holds true that
\be 
E_m (2, \Gamma^J \chi_1)^{1/2}
+
E_1 (2, \Gamma^J \chi_2)^{1/2}
\leq C C_1 \eps h.
\ee
\end{lemma}

Next we look at the energy estimates of $(\chi_1, \chi_2)$ on the flat slices.

\begin{lemma}[Estimates on the flat slices]
\label{lem:EEflat}
With the same assumptions and notations as in Proposition \ref{prop:Va-contraction1}, for all $\Gamma \in A$ we have
\be 
\aligned
\Ecal_1 \big(t, \Gamma^I \chi_1\big)^{1/2}
&\leq C C_1 \eps h,
\qquad
|I| \leq N,
\\
\Ecal_1 \big(t, \Gamma^I \chi_2\big)^{1/2}
&\leq C C_1 \eps h t^\delta,
\qquad
|I| \leq N.
\endaligned
\ee
\end{lemma}

We omit the proof since a very similar argument to the one in the proof of Proposition \ref{prop:Va-EEflat} also applies here. 
Now we turn to estimate the energies of $(\chi_1, \chi_2)$ on the hyperboloidal slices.

\begin{lemma}[Estimates on the hyperboloids]
\label{lem:EEhyper}
For all $|J| \leq N-5$ and $\Gamma \in A$, we have
\be 
\aligned
E_m \big(t, \Gamma^J \chi_1 \big)^{1/2}
&\leq C C_1 \eps h,
\\
E_1 \big(t, \Gamma^J \chi_2 \big)^{1/2}
&\leq C C_1 \eps h s^\delta.
\endaligned
\ee
\end{lemma}

The proof is very similar to the one of Proposition \ref{prop:Va-EEhyper}, and we omit it.

We are in a position to prove Proposition \ref{prop:Va-contraction1}, which further proves Theorem \ref{thm:Va-main}.

\begin{proof}[Proof of Proposition \ref{prop:Va-contraction1}]
By recalling the choice of the constants $C_1, \eps$, and the estimates in Lemma \ref{lem:EEflat} and Lemma \ref{lem:EEhyper}, it is not hard to show
$$
\| (\chi_1, \chi_2) \|_X
\leq {1\over 2} h,
$$ 
which proves Proposition \ref{prop:Va-contraction1}.
\end{proof}

\begin{proof}[Proof of Theorem \ref{thm:Va-main}]
Let $(\lambda_1^{(1)}, \lambda_2^{(1)})$ be the solution to the equations
$$
\aligned
- \Box \lambda_1^{(1)} + m^2 \lambda_1^{(1)}
&= 0,
\\
- \Box \lambda_2^{(1)} + \lambda_2^{(1)}
&= 0,
\\
\big( \lambda_1^{(1)}, \lambda_2^{(1)} \big)(t_0, \cdot)
= (u_0, & v_0),
\qquad
\big( \del_t \lambda_1^{(1)}, \del_t \lambda_2^{(1)} \big)(t_0, \cdot)
= (u_1, v_1),
\endaligned
$$
which can be proved to be an element in space X.

Next we define the sequence
$$
(\lambda_1^{(n)}, \lambda_2^{(n)})
:=
T(\lambda_1^{(n-1)}, \lambda_2^{(n-1)}),
$$
for $n \geq 2$.
Then Proposition \ref{prop:Va-contraction1} tells us that 
$$
\big(\lambda_1^{(n)}, \lambda_2^{(n)} \big)
\xrightarrow{X}
(u, v),
$$
which is the only solution to \eqref{eq:Va-model}.

Finally, we observe that 
$$
|u(t, x)| \lesssim \min \{t^{-1}, m^{-1} t^{-3/2} \}
$$
is equivalent to \eqref{eq:Va-unified}, and hence the proof is complete.
\end{proof}


\subsection{High order estimates on the hyperboloids}

Looking back at the definition of the $\| \cdot \|_X$-norm, there is a lack of derivatives for the energy estimates of the solution $(u, v)$ on the hyperboloids. We will make up that lack in this subsection, which is stated now.

\begin{proposition}\label{prop:Va-high}
Let $(u, v)$ be the solution to the system \eqref{eq:Va-model}, and let all the assumptions in Theorem \ref{thm:Va-main} be true, then we have
\be 
\aligned
s^{- 2 \delta} E_m (s, \Gamma^I u)^{1/2}
+
s^{-1/2 - \delta} E_1(s, \Gamma^I v)^{1/2}
\leq 
C_1 \eps,
\qquad
|I| \leq N,
\qquad
\Gamma \in A.
\endaligned
\ee
\end{proposition}

Note that we already know the solution $(u, v)$ to the system \eqref{eq:Va-model} exists globally, with moreover $\| (u, v) \|_X \leq C_1 \eps$. 
Many of the estimates for $(u, v)$ established in the analysis are ready to use. It is not difficult to show Proposition \ref{prop:Va-high} relying on the standard bootstrap method, where we might harmlessly enlarge $C_1$ or shrink $\eps$, so we omit its proof.


\section*{Acknowledgements}

The author would like to express his sincere thanks to Philippe G. LeFloch (Sorbonne University) for proposing this interesting problem to him, and for many helpful discussions. The author also owes many thanks to Siyuan Ma (Sorbonne University) for his encouragements. The author was supported by the Innovative Training Networks (ITN) grant 642768, entitled ModCompShock.



\end{document}